\documentclass{amsart}

\usepackage{amsmath, amsthm}
\usepackage{amssymb,amscd}
\usepackage{amsfonts}
\usepackage{latexsym}
\usepackage{mathrsfs}
\usepackage{bm}
\usepackage{graphicx}

\theoremstyle{plain}
\newtheorem{theorem}{Theorem}[section]
\newtheorem{cor}[theorem]{Corollary}
\newtheorem{lem}[theorem]{Lemma}
\newtheorem{pro}[theorem]{Proposition}

\newtheorem{Def}[theorem]{Definition}

\newcommand{\hym}{hyperbolic metric}
\newcommand{\hv}{hyperbolic volume}
\newcommand{\ha}{hyperbolic area}
\newcommand{\TS}{Teichm\"{u}ller space}
\newcommand{\Tt}{Teichm\"{u}ller theory}
\newcommand{\WP}{Weil-Petersson}
\newcommand{\kg}{Kleinian group}
\newcommand{\hd}{Hausdorff dimension}

\newcommand{\mc}{mean curvature}
\newcommand{\pc}{principal curvature}
\newcommand{\gc}{Gaussian curvature}
\newcommand{\mcf}{mean curvature flow}
\newcommand{\maxp}{maximum principle}

\newcommand{\RS}{Riemann surface}
\newcommand{\sff}{second fundamental form}
\newcommand{\hf}{height function}
\newcommand{\naf}{nearly Fuchsian}
\newcommand{\af}{almost Fuchsian}
\newcommand{\qf}{quasi-Fuchsian}
\newcommand{\qc}{quasiconformal}
\newcommand{\qcm}{quasiconformal mapping}
\newcommand{\qci}{quasicircle}
\newcommand{\lms}{limit set}

\newcommand{\gradf}{gradient function}

\newcommand{\tm}{three-manifold}
\newcommand{\ms}{minimal surface}
\newcommand{\nf}{normal flow}
\newcommand{\ee}{evolution equation}
\newcommand{\ef}{equidistant foliation}
\newcommand{\is}{incompressible surface}
\newcommand{\ins}{initial surface}
\newcommand{\es}{evolving surface}
\newcommand{\cc}{convex core}

\newcommand{\gbar}{\bar{g}}

\newcommand{\nablabar}{{\overline{\nabla}}}

\newcommand{\n}{\bm{n}}

\newcommand{\C}{\mathbb{C}}
\newcommand{\R}{\mathbb{R}}
\newcommand{\I}{\mathbb{I}}
\renewcommand{\H}{\mathbb{H}}

\newcommand{\Acal}{\mathcal{A}}

\newcommand{\inner}[2]{\langle #1\,,#2\rangle}
\newcommand{\ddl}[2]{\frac{d{#1}}{d{#2}}}

\newcommand{\ppl}[2]{\frac{\partial{#1}}{\partial{#2}}}

\DeclareMathOperator{\hyp}{hyp}%
\DeclareMathOperator{\id}{id}%
\DeclareMathOperator{\vol}{vol}%

\renewcommand{\div}{\mathop{\mathrm{div}}}

\numberwithin{equation}{section}

\allowdisplaybreaks

\def\XXint#1#2#3{{\setbox0=\hbox{$#1{#2#3}{\int}$}
    \vcenter{\hbox{$#2#3$}}\kern-.5\wd0}}

\makeatletter
\def\@citestyle{\m@th\upshape\mdseries}
\def\citeform#1{{\bfseries#1}}
\def\@cite#1#2{{%
  \@citestyle[\citeform{#1}\if@tempswa, #2\fi]}}
\@ifundefined{cite }{%
  \expandafter\let\csname cite \endcsname\cite
  \edef\cite{\@nx\protect\@xp\@nx\csname cite \endcsname}%
}{}
\makeatother

\begin{document}

\title{Mean Curvature Flows in Almost Fuchsian Manifolds}

\author{Zheng Huang}
\address{Department of Mathematics,
The City University of New York,
Staten Island, NY 10314, USA.}
\email{zheng.huang@csi.cuny.edu}

\author{Biao Wang}
\address{Department of Mathematics,
University of Toledo, Toledo, OH 43606, USA.}
\email{biao.wang@utoledo.edu}

\date{December 11, 2009}



\begin{abstract}
An {\af} manifold is a {\qf} hyperbolic {\tm} that contains a closed incompressible {\ms} with {\pc}s everywhere
in the range of $(-1,1)$. In such a hyperbolic {\tm}, the {\ms} is unique and embedded, hence one can
parametrize these {\tm}s by their {\ms}s. We prove that any closed surface which is a graph over any fixed
surface of small {\pc}s can be deformed into the {\ms} via the {\mcf}. We also obtain an upper bound for
the hyperbolic volume of the {\cc} of $M^3$, as well as estimates on the {\hd} of the limit set.
\end{abstract}

\maketitle


\section {Introduction}
Quasi-Fuchsian {\tm} is an important class of complete hyperbolic $3$-manifolds.
In hyperbolic geometry, {\qf} manifolds
and their moduli space, the {\qf} space, have been objects of extensive study in recent decades. Incompressible
surfaces of small {\pc}s play an important role in low dimensional topology (\cite{Rub05}). We denote $M^3$ in
this paper an {\af} manifold: it is a {\qf} manifold which contains a closed incompressible {\ms} $\Sigma$ such
that the {\pc}s of $\Sigma$ are in the range of $(-1,1)$.

The space of {\af} manifolds is a subspace of the {\qf} space, of the same complex dimension $6g-6$, where
$g \ge 2$ is the genus of any closed {\is} in the manifold (\cite {Uhl83}). Understanding the structures of the {\qf}
space is a mixture of understanding the hyperbolic geometry and topology of {\qf} {\tm}s, the deformation of
{\is}s (induce injections between fundamental groups of the surface and the {\tm}), as well as the representation
theory of {\kg}s.

This notion of being {\af} (term coined in \cite{KS07}) was first studied by Uhlenbeck (\cite{Uhl83}), where she
proved several key properties of {\af} manifolds that will be vital in this work: $\Sigma$ is the only incompressible
{\ms} in $M^3$, and $M^3$ admits a foliation of parallel surfaces from $\Sigma$ to both ends.

Our convention of the {\mc} is the sum of the {\pc}s, and we always assume the genus of any {\is} of $M^3$ is at
least two. We also assume $M^3$ is not Fuchsian, or theorems are trivial. Surfaces we encounter in this paper
will always be incompressible, unless otherwise stated.

Somewhat differently from \cite {Wan08} and \cite {HW09}, where we used volume preserving {\mcf} to address
the issue of foliation of constant {\mc} surfaces, in this work, we use the usual {\mcf} to deform a rather arbitrary
closed surface $S_0$ to the {\ms} $\Sigma$. The special geometry of the ambient space ({\af}) allows us to apply
the {\mcf} to drag $S_0$ towards $\Sigma$ rapidly and change the shapes of the {\es}s until the {\mc}s become
zero.

Our main analytical tool is the {\mcf} equation, which has the following form:
\begin{equation}
   \left\{
   \begin{aligned}
      \ppl{}{t}\,F(x,t)&=-H(x,t)\nu(x,t)\ ,\\
      F(\cdot,0)&=F_{0}\ ,
   \end{aligned}
   \right.
\end{equation}
where $H(x,t)$ is the {\mc} of the evolving surface $S(t)$, and all other terms will be made transparent in
section two.

Our main result is that one can deform initial closed graphical surface $S_0$ over a fixed surface
$S$ with $|\lambda_j(S)|<1$ to the unique {\ms} in an {\af} manifold:
\begin{theorem}
Let $M^3$ be {\af} and $S$ ($S$ not necessarily the unique {\ms} of $M^3$) be any closed {\is} with {\pc}s in
$(-1,1)$ everywhere on $S$. Suppose a smooth closed surface $S_0$ in $M^3$ is a graph over $S$: there
is a constant $c_0 > 0$ such that $\langle{\n},{\nu_0}\rangle \geq c_0$, where $\n$ is the unit normal vector
on $S$ and $\nu_0$ is the unit normal vector on $S_{0}$. Then:
\begin{enumerate}
\item
the {\mcf} equation $(1.1)$ with initial surface $S(0) = S_0$ has a long time solution;
\item
the {\es}s $\{S(t)\}_{t \in \R}$ stay smooth and remain as graphs over $S$ for all time;
\item
$\{S(t)\}_{t \in \R}$ converge exponentially to the {\ms} $\Sigma$.
\end{enumerate}
\end{theorem}

We note that we do not require the {\pc}s of the initial surface $S_0$ to be small. In other words, the {\mcf}
$(1.1)$ is very insensitive to the initial data. When $M^3$ is {\af}, there are abundant {\is}s of small {\pc}s, for
instance, every parallel surface from the {\ms} satisfies the {\pc} condition. This theorem is generalized to a
larger class of {\qf} manifolds (see Theorem 6.1) in \S 6.

It is very important to us that we are dealing with graphical surfaces: our estimates rely on the basic fact that
the graph functions behave quite regularly in hyperbolic spaces under evolution equations (see for example
\cite {EH89}, \cite {Unt03}). Additional benefit of having graphical surfaces: these surfaces are naturally
embedded, while it is generally very difficult to prove embeddedness from geometric measure theory.

Since the {\ms} is unique in an {\af} {\tm}, we can use {\ms}s to parametrize the space of {\af} manifolds. We
obtain topological and geometric information about $M^3$ from data of $\Sigma$. In particular, the {\cc} of a
{\qf} {\tm} is a crucial part of the manifold. It is the smallest convex subset of a {\qf} manifold that
carries its fundamental group. From the point of view of hyperbolic geometry, the {\cc} contains all the
geometrical information about the {\qf} {\tm} itself (see for instance, \cite {AC96}, \cite{Bro03}). As a direct
application, when $M^3$ is {\af}, we obtain an explicit upper bound for the {\hv} of $C(M^3)$,
in terms of the maximum {\pc} on the {\ms} $\Sigma$:

\begin{theorem}
If $M^3$ is {\af}, and let $\lambda_0 = \max\limits_{x \in \Sigma}\{|\lambda(x)|\}$, then
\begin{align}
    \vol(C(M^3))
       &\le \Acal_{\hyp}\left({\frac{\lambda_0}
            {1-\lambda_0^2}}+{\frac{1}{2}}
            \log{\frac{1+\lambda_0}
            {1-\lambda_0}}\right) \notag \\
       &= \Acal_{\hyp} \left(2\lambda_0 +
          {\frac{4}{3}}\lambda_0^3 +
          O(\lambda_0^5)\right),
\end{align}
where $\Acal_{\hyp} = 2\pi(2g-2)$ is the {\ha} of any closed {\RS} of genus $g>1$.
\end{theorem}

We also obtain an upper bound for the {\hd} of the limit set $\Lambda_{\Gamma}$ of $M^3$, in terms of
$\lambda_0$ as well:
\begin{theorem}
If $M^3$ is {\af}, and let $\lambda_0 = \max\limits_{x \in \Sigma}\{|\lambda(x)|\}$, then the {\hd}
$D(\Lambda_{\Gamma})$ of the limit set $\Lambda_{\Gamma}$ for $M^3 = \H^3/\Gamma$ satisfies
\begin{equation}
D(\Lambda_{\Gamma}) < 1+ \lambda_0^2.
\end{equation}
\end{theorem}

For $\lambda_0$ close to $0$, Theorems 1.2 and 1.3 measure how close $M$ is to being Fuchsian.
In \cite{Bro03}, Brock showed the {\hv} of the {\cc} is quasi-isometric to the {\WP} distance between
conformal boundaries of $M^3$ in {\TS}. We showed (\cite{GHW10}) that the area functional of the {\ms}
is a potential for the {\WP} metric in {\TS}.

The differential equations of the evolution of hypersurfaces by their {\mc} have been studied extensively in
various ambient Riemannian manifolds (see for instance \cite{Bra78}, \cite{Hui84}, \cite{Hui86},
\cite{Hui87a}, and many others). Our study is of two-folds: the setting of {\qf} hyperbolic three-manifolds
provides effective barrier surfaces (Theorem 4.5) for the flow, and the convergence of the flow leads to the
discovery of incompressible {\ms}s in the three-manifolds.

\subsection*{Plan of the paper}
We provide necessary background material in \S 2, especially the {\af} manifolds, the {\ef}, and the {\mcf}.
Section \S 3 is focused on the geometry of {\af} manifolds, where we prove the Theorem 1.2 (volume estimate)
and Theorem 1.3 ({\hd} estimate). Special geometry of these three-manifolds plays important roles in our
analysis of the {\mcf}. We prove the Theorem 1.1 in next two sections: \S 4 (uniform bound for the
{\sff} and {\mc}), \S 5 (long time existence, convergence and uniqueness of the limit).  We apply our technique to
more general cases in section \S 6, where we also show the mean convexity is preserved under the flow.

\subsection*{Acknowledgements}
The authors are grateful to Ren Guo and Zhou Zhang for many helpful discussions. We also thank Dick Canary 
and Jun Hu for their suggestions regarding the {\hd} of the limit set. The research of the first named author is 
partially supported by a PSC-CUNY grant.


\section{Preliminaries}
In this section, we fix our notations, and introduce some preliminary facts that will be used in this paper.


\subsection{Quasi-Fuchsian manifolds}
For detailed reference on {\kg}s and low dimensional topology, one can go to \cite{Mar74} and \cite{Thu82}.

The universal cover of a hyperbolic {\tm} is $\H^3$, and the deck transformations induce a representation of the
fundamental group of the manifold in $Isom(\H^3)= PSL(2,\C)$, the (orientation preserving) isometry group of
$\H^3$. A subgroup $\Gamma \subset PSL(2,\C)$ is called a {\em Kleinian group} if $\Gamma$ acts on $\H^{3}$
properly discontinuously. For any {\kg} $\Gamma$, $\forall\,p\in\H^{3}$, the orbit set
\begin{eqnarray*}
   \Gamma(p)=\{\gamma(p)\ |\ \gamma\in \Gamma\}
\end{eqnarray*}
has accumulation points on the boundary $S^{2}_\infty=\partial\H^{3}$, and these points are the {\em limit points}
of $\Gamma$, and the closed set of all these points is called the {\em limit set} of $\Gamma$, denoted by
$\Lambda_{\Gamma}$. In the case when $\Lambda_{\Gamma}$ is contained in a circle $S^1 \subset S^2$, the
quotient $M^3 =\H^{3}/\Gamma$ is called {\it Fuchsian}, and $M^3$ is isometric to a product space $S \times \R$. If
the limit set $\Lambda_{\Gamma}$ lies in a Jordan curve, $M^3 =\H^{3}/\Gamma$ is called {\bf {\qf}}, and it is
topologically $S \times \R$, where $S$ is a closed surface of genus $g$ at least two. It is clear that a {\qf} manifold is
quasi-isometric to a Fuchsian manifold. The space of such three manifolds, the {\qf} space of genus $g$ surfaces,
is a complex manifold of dimension of $6g-6$, which has very complicated structures.

Finding {\ms}s in negatively curved manifolds is a problem of fundamental importance. The basic results are
due to Schoen-Yau (\cite {SY79}) and Sacks-Uhlenbeck (\cite{SU82}), and their results can be applied to the
case of {\qf} {\tm}s: any {\qf} manifold contains at least one incompressible {\ms}. In the case of {\af}, the {\ms} is
unique (\cite{Uhl83}). On the other hand, there are many {\qf} manifolds that admit many {\ms}s (\cite{Wan10}).

An essential problem in hyperbolic geometry and complex dynamics is to study the {\hd}
$D(\Lambda_{\Gamma})$ of the limit set $\Lambda_{\Gamma}$ associated to $M^3$. This problem also intimately
related to understanding the lower spectrum theory of hyperbolic 3-manifolds (\cite{Sul87, BC94}). In the case of
Fuchsian manifolds, $\Lambda_{\Gamma}$ is a round circle and $D(\Lambda_{\Gamma}) = 1$. When $M^3$ is
{\qf} but not Fuchsian, as throughout this paper, we have $1 < D(\Lambda_{\Gamma}) < 2$. There is a rich
theory of {\qcm} and its distortion in {\hd}, area and other quantities (see for instance \cite {GV73, LV73}).

The following lemma is the well-known Hopf's maximum principle for tangential hypersurfaces in Riemannian
geometry.
\begin{lem} [\cite{Hop89}]
Let $S_{1}$ and $S_{2}$ be two hypersurfaces in a Riemannian manifold which intersect at a common point
tangentially. If $S_{2}$ lies in positive side of $S_{1}$ around the common point, then $H_{1} \le H_{2}$, where
$H_{i}$ is the {\mc} of $S_{i}$ at the common point for $i=1,2$.
\end{lem}

\subsection{Almost Fuchsian manifolds}
Throughout the paper, $M^3$ is an {\af} {\tm}: the {\pc}s of the {\ms} $\Sigma$ are in the range of $(-1,1)$.

The induced metric on an {\is} $S$ is $g_{ij}(x)=e^{2v(x)}\delta_{ij}$, where $v(x)$ is a smooth function on $S$,
and while the {\sff} $A(x)=[h_{ij}]_{2\times{}2}$, here $h_{ij}$ is given by, for $1\leq{}i,j\leq{}2$,
\begin{equation*}
   h_{ij}=\langle{\nablabar_{e_{i}}\nu}, {e_{j}}\rangle
         =-\langle{\nablabar_{\nu}e_{i}},{e_{j}}\rangle\ ,
\end{equation*}
where $\{e_{1},e_{2}\}$ is a basis on $S$, and $\nu$ is the unit normal field on $S$, and $\nablabar$ is the
Levi-Civita connection of $(M^3,\gbar_{\alpha\beta})$.

We add a bar on top for each quantity or operator with respect to $(M^3,\gbar_{\alpha\beta})$.

Let $\lambda_1(x)$ and $\lambda_2(x)$ be the eigenvalues of $A(x)$. They are the {\pc}s of $S$, and we denote
$H(x) =\lambda_1(x)+ \lambda_2(x)$ as the {\mc} function of $S$.

Let $S(r)$ be the family of equidistant surfaces with respect to $S$, i.e.
\begin{equation*}
   S(r)=\{\exp_{x}(r\nu)\ |\ x\in{}S\}\ ,
   \quad{}r\in(-\varepsilon,\varepsilon)\ .
\end{equation*}
The induced metric on $S(r)$ is denoted by $g(x,r) = g_{ij}(x,r)$, and the {\sff} is denoted by
$A(x,r)=[h_{ij}(x,r)]_{1\leq{}i,j\leq{}2}$. The {\mc} on $S(r)$ is thus given by $H(x,r)=g^{ij}(x,r)h_{ij}(x,r)$.
\begin{lem}[\cite {Uhl83}, \cite{HW09}]
The induced metric $g(x,r)$ on $S(r)$ has the form
\begin{equation}\label{eq:metric of S(r)}
   g(x,r)=e^{2v(x)}[\cosh(r)\I+\sinh(r)
   e^{-2v(x)}A(x)]^{2}\ ,
\end{equation}
where $r \in (-\varepsilon,\varepsilon)$.

The {\pc}s of $S(r)$ are given by
\begin{equation}\label{pc4ef}
   \mu_{j}(x,r)=\frac{\tanh(r)+\lambda_{j}(x)}{1+\lambda_{j}(x)\tanh(r)}\ ,\qquad j=1,2\ .
\end{equation}
and mean curvature is
\begin{equation} \label{mc4ef}
   H(x,r)=\frac{2(1+\lambda_{1}\lambda_{2})\tanh(r)+(\lambda_{1}+\lambda_{2})(1+\tanh^{2}(r))}
   {1+(\lambda_{1}+\lambda_{2})\tanh(r)+\lambda_{1}\lambda_{2}\tanh^{2}(r)}\ .
\end{equation}
\end{lem}
Clearly, when $|\lambda_{j}(x)| < 1$ for $j=1,2$, the metrics $g(x,r)$ are of no singularity for all $r \in \R$ and
therefore $\{S(r)\}_{r\in\R}$ forms a foliation of surfaces parallel to $S$, called the {\it {\ef}} or the {\it {\nf}}. In
what follows, we have $|\lambda_{j}(S)| < 1$ for $j=1,2$.

Since $M^3$ is {\af}, then we have $|\lambda_{j}(x)| < 1$ for $j=1,2$ and $x \in \Sigma$. The {\ef} from the
{\ms} $\Sigma$ is then denoted by $\{\Sigma(r)\}_{r\in \R}$. Each fiber $\Sigma(r)$ satisfies the {\pc}s lie in
$(-1,1)$. We may assume the reference surface $S = \Sigma$ to simplify the exposition in the first part of the
proof of Theorem 4.5.

The existence of {\ef} is a remarkable property of {\af} manifolds. Recently, very interesting applications of
{\af} manifolds have been explored in the content of mathematical physics, see for example, \cite {KS07},
\cite{KS08} and others.

\subsection{Mean curvature flow}

Let $F_{0}: S \to{}M^3 = S \times \R$ be the immersion of an {\is} $S$ in $M^3$. Without loss of generality, we
assume that $S_{0}=F_{0}(S)$ is contained in the positive side of the {\ms} $\Sigma$, and is a graph over
$\Sigma$ with respect to $\n$, i.e., $\langle{\n},{\nu_0}\rangle \geq c_0>0$, here $\n$ is the unit normal
vector on $\Sigma$, $\nu_0$ is the unit normal vector on $S_{0}$ and $c_0$ is a constant depending only
on $S_{0}$.

We consider a family of immersions of surfaces in $M^3$,
\begin{equation*}
   F:S\times[0,T)\to{}M^3\ , \quad{}0\leq{}T\leq\infty
\end{equation*}
with $F(\cdot,0)=F_{0}$. For each $t\in[0,T)$, $S(t)=\{F(x,t)\in{}M^3\ |\ x\in{}S\}$ is the evolving surface
at time $t$, and $H(x,t)$ its {\mc}.

The {\mcf} equation is given by, as in $(1.1)$:
\begin{equation*}
   \left\{
   \begin{aligned}
      \ppl{}{t}\,F(x,t)&=-H(x,t)\nu(x,t)\ ,\\
      F(\cdot,0)&=F_{0}\ ,
   \end{aligned}
   \right.
\end{equation*}
Here $\nu(x,t)$ is the normal vector on $S(t)$ with $-\nu$ points to the reference surface $S$.

This is a parabolic equation, and Huisken proved the short-time existence of the solutions to
$(1.1)$, and initial compact surface quickly develops singularities along the flow, moreover, he
showed the blow-up of the norm of the {\sff}s if the singularity occurs in finite time.
\begin{theorem}[\cite {Hui84}]
If the initial surface $S_0$ is smooth, then the equation $(1.1)$ has a smooth solution on some
maximal open time interval $0\leq{}t<T$, where $0<T \leq \infty$. If $T<\infty$, then
$|A|_{\max}(t)\equiv\max\limits_{x\in{}S}|A|(x,t)\to\infty$\, as $t\to{}T$.
\end{theorem}

\section{Geometry of Almost Fuchsian Manifolds}
In this section, we wish to obtain information about the {\af} manifold $M^3$ via its unique {\ms} $\Sigma$. We derive
several geometrical properties on the {\ef} $\{\Sigma(r)\}_{r \in \R}$ in \S 3.1, and in \S 3.2, we obtain explicit upper
bounds for the {\hv} of the {\cc} of $M^3$. Proposition 3.4 is particularly useful both in the proof of Theorem 4.5 and
Theorem 1.2.
\subsection{Some estimates on $\{\Sigma(r)\}_{r \in \R}$}

We record {\pc}s of the {\ms} $\Sigma$ by $\pm \lambda$ and $-1 < \lambda < 1$, and $|S|$ be the area for any
{\is} $S$ (with respect to the induced metric), and $\Acal_{\hyp} = 2\pi(2g-2)$ the {\ha} of the surface $S$.

We start with a well-known estimate which implies the area of the {\ms} under the induced metric from ambient space is
comparable to that of the hyperbolic area, with universal constants. We only include a proof for the sake of completeness.
\begin{pro}
$\Acal_{\hyp}/2 < |\Sigma| < \Acal_{\hyp}$.
\end{pro}
\begin{proof}
We apply the Gauss equation:
\begin{equation*}
   K(\Sigma) = -1 + \det(A),
\end{equation*}
where $K(\Sigma)$ is the {\gc} of $\Sigma$ and $A$ is the {\sff} of $\Sigma$.

Thus we have
\begin{equation*}
-K(\Sigma) = 1 - \det(A) = 1 + \lambda^2.
\end{equation*}
We integrate this on $\Sigma$, applying the Gauss-Bonnet theorem since $\Sigma$ is incompressible, to find
\begin{equation*}
   |\Sigma| < |\Sigma| + \int_{\Sigma}\lambda^2 = \Acal_{\hyp} < 2|\Sigma|.
\end{equation*}
\end{proof}

We want to estimate the area of each parallel surface in the {\ef} $\{\Sigma(r)\}_{r \in \R}$:
\begin{pro}
For all $-\infty < r < +\infty$, we have
\begin{equation*}
   {\frac{\Acal_{\hyp}}{2}} < |\Sigma| \le |\Sigma(r)| \le |\Sigma |\cosh^2(r) < \Acal_{\hyp}\cosh^2(r)
\end{equation*}
\end{pro}

\begin{proof}
The area element of $\Sigma (r)$ is given by
\begin{equation}
d\mu(r) = (\cosh^2(r) - \lambda^{2}(x)\sinh^{2}(r))d\mu,
\end{equation}
where $d\mu$ is the area element for the {\ms} $\Sigma$.

We can now compute the surface area:
\begin{align}
  |\Sigma(r)| &= \int_{\Sigma} (\cosh^2(r) - \lambda^{2}(x)\sinh^{2}(r))d\mu  \nonumber \\
         &= \cosh^{2}(r) |\Sigma| - \sinh^{2}(r) \int_{\Sigma}\lambda^{2}(x)d\mu \nonumber \\
         &= \cosh^{2}(r) |\Sigma| -
            \sinh^{2}(r)(\Acal_{\hyp} - |\Sigma|)\nonumber \\
         &= |\Sigma|(\cosh^{2}(r) + \sinh^{2}(r)) - \sinh^{2}(r)\Acal_{\hyp}\nonumber \\
         &= |\Sigma| + \sinh^{2}(r)(2|\Sigma|-\Acal_{\hyp}) \ ,
\end{align}
Here we used the identity
\begin{equation}
\int_{\Sigma}\lambda^2 = \Acal_{\hyp} - |\Sigma|.
\end{equation}
The estimates then follows from the Proposition 5.1 and $(5.2)$.
\end{proof}

We also need the estimates on the {\mc}s of $\Sigma(r)$:
\begin{pro}
$|H(\Sigma(r))| \le 2|\tanh(r)| <2$.
\end{pro}
\begin{proof}
We only prove the part when $r > 0$. The {\mc} of the
surface $\Sigma(r)$ is given by the formula $(2.3)$:
\begin{equation*}
   H(x,r)=\frac{2(1-\lambda^{2}(x))\tanh(r)}
   {1-\lambda^{2}(x)\tanh^{2}(r)}\ ,
   \quad\forall\,x\in \Sigma\ .
\end{equation*}
An easy calculation shows
\begin{equation*}
   H(x,r) - 2\tanh(r) =
   {\frac{2\lambda^2\tanh(r)(\tanh^{2}(r)-1)}
   {1-\lambda^2\tanh^{2}(r)}} \le 0.
\end{equation*}
\end{proof}
\subsection{Upper bound for the {\cc} volume}

We obtain an upper bound for the {\hv} of the {\cc} $C(M^3)$ in
this subsection, in terms of the
maximum, $\lambda_0$, of the {\pc} function on the {\ms} $\Sigma$.

We recall from formula $(2.2)$, that the {\pc}s of the surface
$\Sigma(r)$, $\mu_{1}(x,r)$ and
$\mu_{2}(x,r)$ are increasing functions of $r$ for any fixed
$x \in \Sigma$, and they approach
$\pm 2$ as $r \rightarrow \pm\infty$. We also have
$\mu_1(x,r) \le \mu_2(x,r)$ for fixed $r$ and $x$.

We are particularly interested in two critical cases: the values of $r$
when $\mu_1(x,r) = 0$ and
$\mu_2(x,r) = 0$. Elementary algebra shows:
\begin{pro} If we denote
\begin{equation*}
   r_0 = {\frac{1}{2}}\log{\frac{1+\lambda_0}{1-\lambda_0}}\ ,
\end{equation*}
where $\lambda_0 = \max\limits_{x \in S}\{\lambda(x)\}$, then $r_0$ is the least value of $r$ such that
$\mu_1(r,x) > 0$ for all $r > r_0$, and $-r_0$ is the largest value for $\mu_2(r,x) < 0$ such that
$\mu_2(r,x)<0$ for all $r<-r_0$.
\end{pro}
This Proposition tells us when the parallel surfaces in the {\ef} $\{\Sigma(r)\}_{r \in \R}$ become
convex, hence by the definition of the {\cc}, provides an upper bound for the size of the {\cc}.

We denote the region of $M^3$ bounded between the surfaces $\Sigma(-r_0)$ and $\Sigma(r_0)$
by $M(r_0)$, and then the {\cc} $C(M^3)$, is contained in $M(r_0)$. Since $\{\Sigma(r)\}_{r \in \R}$
foliates $M^3$, we can compute the {\hv} of the region $M(r_0)$ by
\begin{align}
   \vol(M(r_0))
       &= \int_{-r_0}^{r_0}|\Sigma(r)|dr \nonumber \\
       &= 2r_{0}|\Sigma| + (2|\Sigma| - \Acal_{\hyp})
          \int_{-r_0}^{r_0}\sinh^{2}(r)dr \nonumber \\
       &= 2r_{0}|\Sigma| + (2|\Sigma| - \Acal_{\hyp})
          \left({\frac{1}{2}}\sinh(2r_0) -
          r_0\right)  \nonumber \\
       &= |\Sigma|\sinh(2r_0) - \Acal_{\hyp}
          \left({\frac{1}{2}}\sinh(2r_0) - r_0\right)
\end{align}
Applying the Proposition 3.2, we obtain the following:
\begin{theorem} The {\hv} of $C(M^3)$ is bounded by:
\begin{eqnarray}
    \vol(C(M^3)) &\le& \Acal_{\hyp}(\cosh{}r_{0}
                    \sinh{}r_{0} + r_0)\nonumber \\
               &=& \Acal_{\hyp}\left({\frac{\lambda_0}
                  {1-\lambda_0^2}}+
                  {\frac{1}{2}}\log{\frac{1+\lambda_0}
                  {1-\lambda_0}}\right).
\end{eqnarray}
\end{theorem}

This estimate in the Theorem 3.5 can also be obtained via an application of the
Cauchy-Schwarz inequality and the Propositions 3.2 and 3.3.

When $r_0$, or equivalently, $\lambda(x) =0$ for all $x \in \Sigma$, this is the case of
$M$ being Fuchsian, in which case, the hyperbolic volume of $C(M^3)$ is zero. We want to
measure how the volumes vary for small $\lambda_0$. From Taylor series expansion we have
\begin{cor}
For small $\lambda_0$, we have the following expansion:
\begin{equation*}
   \vol(C(M^3)) \le \Acal_{\hyp}\left(2\lambda_0 + {\frac{4}{3}}\lambda_{0}^{3} + O(\lambda_0^5)\right).
\end{equation*}
\end{cor}
\subsection{{\hd} of the limit set}
We denote $C_1(M^3)$ the hyperbolic radius one neighborhood of the {\cc} $C(M^3)$ in $M^3$. An
easy calculation from $(3.4)$ gives us
\begin{equation*}
\vol(C_1(M^3)) \le \vol(M(r_0+1)) \le  \Acal_{\hyp}\left({\frac{1}{2}}\sinh(2r_0+2) - r_0 -1\right),
\end{equation*}
where  $r_0 = {\frac{1}{2}}\log{\frac{1+\lambda_0}{1-\lambda_0}}$. Therefore we have
\begin{equation}
\vol(C_1(M^3)) \le  \Acal_{\hyp}\left({\frac{1}{2}}\sinh(\log{\frac{1+\lambda_0}{1-\lambda_0}}+2) -
{\frac{1}{2}}\log{\frac{1+\lambda_0}{1-\lambda_0}} -1\right).
\end{equation}
Since {\qf} manifolds are geometrically finite and infinite volume, and we assume $M^3$ is not Fuchsian,
a direct application of the main theorem from Burger-Canary (\cite{BC94}) gives:
\begin{pro}
Let $M^3$ be {\af}, and $\Lambda_0(M^3)$ be the bottom of the $L^2$-spectrum of $-\Delta$ on
$M^3$, and $D(\Lambda_{\Gamma})$ be the {\hd} of the limit set $\Lambda_{\Gamma}$ of $M^3$. Then
we have
\begin{enumerate}
\item
\begin{equation}
\Lambda_0(M^3) \ge  {\frac{K_3}{\Acal_{\hyp}^2\left({\frac{1}{2}}\sinh(\log{\frac{1+\lambda_0}{1-\lambda_0}}+2) -
{\frac{1}{2}}\log{\frac{1+\lambda_0}{1-\lambda_0}} -1\right)^2}}.
\end{equation}
\item
\begin{equation}
D(\Lambda_{\Gamma}) \le  2-{\frac{K_3}{\Acal_{\hyp}^2\left({\frac{1}{2}}\sinh(\log{\frac{1+\lambda_0}{1-\lambda_0}}+2) -
{\frac{1}{2}}\log{\frac{1+\lambda_0}{1-\lambda_0}} -1\right)^2}}.
\end{equation}
\end{enumerate}
Here $K_3$ can be chosen such that $K_3 > 10^{-11}$.
\end{pro}
We note that while the volume estimate of the {\cc} of $M^3$ (Theorem 3.5) is effective for small maximal
{\pc} $\lambda_0$ of the {\ms} $\Sigma$, above estimates on $\Lambda_0(M^3)$ and $D(\Lambda_{\Gamma})$
are not as effective. To obtain a better estimate, we consider the {\lms} $\Lambda_{\Gamma}$ of $M^3$ as a
$k$-{\qci} (an image of a circle under a $k$-{\qcm}).

A {\it $k$-{\qc}} mapping $f$ is a homemorphism of planar domains, locally in the Sobolev class $W_2^1$
such that its {\it Beltrami coefficient} $\mu_f = \frac{\bar\partial f}{\partial f}$ has bounded $L^{\infty}$ bound:
$\|\mu_f\| \le k < 1$. One can visualize $f$ infinitesimally maps a round circle to an ellipse with a bounded
 dilatation $K = \frac{1+k}{1-k}$, where $k \in [0,1)$. Clearly, the mapping $f$ is conformal when $k =0$. This is an
 important generalization of conformal maps. The study of {\qcm}s is a major
 breakthrough of geometric function theory behind Teichm\"{u}ller's insight and Ahlfors-Bers' revival of {\Tt}.

\begin{proof}[Proof of Theorem 1.3] Let $\Sigma \subset{}M^3$ be the
incompressible {\ms} with principal curvatures in $(-1,1)$. We know that the
normal bundle over $\Sigma$ is trivial, i.e., the geodesics perpendicular
to $\Sigma$ are disjoint from each other. Therefore, any point $p\in{}M$
can be represented by $p=(x,r)$, here $x$ is the projection of $p$
to $\Sigma$ along the geodesic which passes through $p$ and is
perpendicular to $\Sigma$, and $r$ is the (signed) distance between $p$
and $x$.

Now we can construct aFuchsian $3$-manifold $N=\Sigma \times\R$ as follows. Suppose that the
induced metric on $\Sigma \subset{}M^3$ is given by $g(x)=e^{2v(x)}I$, here
$v(x)$ is a smooth function defined on $\Sigma$ and $I$ is the
$2\times{}2$ identity matrix. The metric $\bar\rho$ of $N$ is given
by
\begin{equation*}
   \bar\rho(x,r)=
   \begin{pmatrix}
      \rho(x,r) & 0 \\
      0         & 1
   \end{pmatrix}\ ,
\end{equation*}
here $\rho(x,r)=e^{2v(x)}\cosh{r}I$. By the construction, it's easy
to know that the surface $\Sigma\times\{0\}$ is totally geodesic.
Similarly, any point $q\in{}N$ can be represented by $q=(y,s)$, here
$y$ is the projection of $q$ to $\Sigma\times\{0\}$ and $s$ is the
distance between $q$ and $y$.

Then we may define a map $\varphi:N\to M^3$ by
$\varphi(x,r)=(x,r)$ for $(x,r)\in N$.
By the result in \cite[p. 162]{Uhl83}, the map $\varphi$ is a quasi-isometry.
Lift $\varphi$ to the map $\tilde\varphi:\H^3\to\H^3$,
then $\tilde\varphi$ is also a quasi-isometry. By the results in
\cite[Theorem 9]{Geh62}, \cite[Theorem 12.1]{Mos68},
\cite[Corollary 5.9.6]{Thu82} and
\cite[Theorem 3.22]{MT98}, $\tilde\varphi$ can be extend to an
automorphism
\begin{equation}
   \breve\varphi:\H^3\cup\widehat{\C}\to\H^3\cup\widehat{\C}
\end{equation}
such that the restriction $\breve\varphi|{\widehat\C}=:f$ is a
quasiconformal mapping. In particular, $f$ maps $S^{2}_{\pm}$ to
$\Omega_{\pm}(\Gamma)$, here
$S^{2}_{\pm}=S^{2}_{\infty}\setminus{}S^{1}$ are hemispheres such
that $\partial{}S^{2}_{+}=S^{1}=\partial{}S^{2}_{-}$, and
$f(S^{1})=\Lambda_\Gamma$ respectively.

We claim that $f|{S^{2}_{+}}:S^{2}_{+}\to\Omega_{+}(\Gamma)$ is a
$k$-quasiconformal mapping, with the dilatation $K = \frac{1+k}{1-k}$, and 
\begin{equation*}
   K<\frac{1+\lambda_{0}}{1-\lambda_{0}}\ .
\end{equation*}
Let $\Pi$ be the lift of the totally geodesic minimal surface
$\Sigma \times\{0\}\subset{}N$, and let $\widehat{\Sigma}$ be the lift of the
surface $\Sigma\subset{}M^3$. Recall that the identity map between $\Pi$
and $\widehat{\Sigma}$ is an isometry, and we can define hyperbolic Gauss
maps $G'_{+}:\Pi\to{}S^{2}_{+}$ and 
$G''_{+}:\widehat{\Sigma} \to\Omega_{+}(\Gamma)$ (\cite{Eps86}) such that
we have the following commutative diagram
\begin{equation*}
\begin{CD}
   S^{2}_{+} @>{f}>> \Omega_{+}(\Gamma) \\
    @A{G'_{+}}AA    @AA{G''_{+}}A \\
   \Pi @>{\id}>> \widehat{\Sigma}
\end{CD}
\end{equation*}
Since $G'_{+}$ is a conformal mapping and $\id$ is an isometry, we
know that $G''_{+}\circ{}f$ is also a conformal mapping. By
Proposition 5.1 and Corollary 5.3 in \cite{Eps86}, $(G''_{+})^{-1}$
is a $k$-quasiconformal mapping, so is $f$.

In particular, $\Lambda_{\Gamma}=f(S^{1})$ is a $k$-quasicircle.
Recently, Smirnov (\cite{Smi09}) proved Astala's conjecture
(\cite{Ast94}): the Hausdorff dimension of a $k$-quasicircle is at
most $1+k^2$. Now
\begin{equation*}
k = \frac{K -1}{K+1} < \lambda_0.
\end{equation*}
This proves the Theorem 1.3.

 \end{proof}
\section{Proof of Theorem 1.1: uniform bounds}
The next two sections are set up to prove the Theorem 1.1. The strategy is standard, though
technical: we establish the uniform bounds of the square norm of the {\sff}, $|A|^2$, and the square of
the {\mc}, $H^2$, as well as their derivatives; we bound the evolving surfaces $S(t)$ in
a compact region of $M^3$ by a {\hf} estimate; these bounds enables us to extend the solution
of $(1.1)$ beyond its maximal finite time interval; we show exponential convergence to prove the
uniqueness of the limiting surface.

\subsection{Some evolution equations}

In this subsection, we collect and derive a number of evolution equations of some quantities and operators on
$S(t)$, $t\in[0,T)$, which are involved in our calculations. These quantities and operators are:
\begin{enumerate}
   \item
   the induced metric of $S(t)$: $g(t)=\{g_{ij}(t)\}$;
   \item
    the {\sff} of $S(t)$: $A(\cdot,t)=\{h_{ij}(\cdot,t)\}$;
   \item
   the {\mc} of $S(t)$ with respect to the normal vector pointing to $\Sigma$: $H(\cdot,t)=g^{ij}h_{ij}$;
   \item
   the square norm of the {\sff} of $S(t)$:
         \begin{equation*}
            |A(\cdot,t)|^{2}=g^{ij}g^{kl}h_{ik}h_{jl}\ ;
         \end{equation*}
   \item
   the covariant derivative of $S(t)$ is denoted by $\nabla$ and the Laplacian on $S(t)$ is given by
         $\Delta=g^{ij}\nabla_{i}\nabla_{j}$.
\end{enumerate}

We start with some standard evolution equations:
\begin{lem} (\cite{Hui86})
The evolution equations of the induced metric $g_{ij}$, the normal vector field $\nu$, and the area element
$d\mu$ are given by
\begin{align}
   \ppl{}{t}\,g_{ij}&=-2Hh_{ij}\ ,\\
   \ppl{}{t}\,\nu&=\nabla H\ ,\\
   \ppl{}{t}\,d\mu&=-H^{2}d\mu\,
\end{align}
\end{lem}
It is clear from $(4.3)$ that the {\mcf} decreases the areas of the {\es}s.

We also need the following evolution equations for the {\mc} $H(\cdot,t)$, and the square norm of the {\sff}
$|A(\cdot,t)|^2$:
\begin{lem}
\begin{align}
      \left(\ppl{}{t}-\Delta\right)\,H=
        &\,H(|A|^{2}-2)\ ,\\
            \left(\ppl{}{t}-\Delta\right)\,|A|^{2}=&\, -2|\nabla{}A|^{2}+2|A|^{2}(|A|^2-2)+4(2|A|^2 -H^2)\ .
\end{align}
\end{lem}
\begin{proof}
These equations are deduced for general Riemannian manifolds in \cite{Hui86}. In our case of hyperbolic {\tm},
the ambient space $M^3$ has constant sectional curvature $-1$, and the Ricci curvature $Ric(\nu,\nu) = -2$ for
any unit normal vector $\nu$.

The lemma then follows from combining these explicit curvature conditions and curvature equations
$\bar{R}_{3i3j} = -g_{ij}$, as well as the well-known Simons' identity (see \cite{Sim68} or \cite{SSY75}), satisfied
by the {\sff} $h_{ij}$:
\begin{equation*}
   \Delta{}h_{ij}=\nabla_{i}\nabla_{j}H-(|A|^{2}-2)h_{ij}+ Hh_{il}h_{lj}+Hg_{ij}\ .
\end{equation*}
\end{proof}

Our estimates will also involve the {\it height function} $u(x,t)$ and the {\it gradient function} $\Theta(x,t)$
on {\es}s $S(t)$:
\begin{align}
   u(x,t)&=\ell(F(x,t))\\
      \label{eq:gradient function}
   \Theta(x,t)&=\langle{\nu(x,t)},{\n}\rangle\ .
\end{align}
Here $T_{\max}$ is the right endpoint of the maximal closed time interval on which the solution to exists, and
$\ell(p) = \pm dist(p,S)$ for all $p \in M^3$, the distance to the reference surface $S$.

We always have $\Theta (x,t) \in [0,1]$. It is clear that the surface $S(t)$ is a graph over $S$ if
$\Theta > 0$ on $S(t)$.

Our main assumption on the initial surface $S_0$ is that $\Theta(x,0) \ge c_0 > 0$. Geometrically, one can
view the initial surface has bounded geometry.

\begin{lem}(\cite {Bar84}, \cite{EH91}) The evolution equations of $u$ and $\Theta$ have the following forms
\begin{align}
   \ppl{}{t}\,u
        =&\,-H\Theta,\\
               \label{eq:evolution of gradient}
    \left(\ppl{}{t}-\Delta\right)\Theta=
       &\,(|A|^{2}-2)\Theta+\n(H_{\n})
          -H\inner{\nablabar_{\nu}\n}{\nu}\ .
\end{align}
Here $\nablabar$ is the gradient operator with respect to the {\hym}, $\div$ is the divergence on $S(t)$, and
$\n(H_{\n})$ is the variation of {\mc} function of $S(t)$ under the deformation vector field $\n$.
\end{lem}

\subsection{Positivity of $\Theta(\cdot,t)$}
In this subsection, we establish the fact that {\es}s $S(t)$ remain as the graphs over $S$, provided the initial
surface $S_0$ is a graph over $S$. This step will be important in the proof of
Theorem 4.5 where we show all evolving surfaces stay in a compact region along the flow.
\begin{lem} \label{bd4gf}
If $\Theta(\cdot,0)\geq{}c_0>0$, where $c_0$ is any positive constant only depending on the initial surface,
then $\Theta (\cdot, t)>0$ for $t \in [0,T)$.
\end{lem}
\begin{proof}Let
\begin{equation*}
   \Theta_{\min}(t)=\min_{x\in{}S}\Theta(x,t)\ .
\end{equation*}
We estimate the terms in the evolution equation $(4.9)$ of $\Theta$, starting with the expression
$\n(H_{\n})$, from (\cite[Eq. (2.10)]{Bar84}):
\begin{equation*}
   |\n(H_{\n})|\leq{}c_{1}(\Theta^{3}+\Theta^{2}|A|)\ ,
\end{equation*}
for some $c_1 > 0$.

We also have the following estimate from (\cite[Page 187]{Eck03})
\begin{equation*}
   |\inner{\nablabar_{\nu}\n}{\nu}|\leq{}c_{2}\Theta^{2}\ ,
\end{equation*}
where $c_{2}=\|\nablabar\n\| > 0$.

Collecting these estimates, and we derive from the equation $(4.9)$:
\begin{eqnarray*}
   \ddl{}{t} \Theta_{\min} &\ge& (|A|^{2}-2)\Theta_{\min}-
             c_{1}(\Theta_{\min}^{3}+|A|\Theta_{\min}^{2})-c_{2}|H|\Theta_{\min}^{2}\\
      &\ge&\big((|A|^{2}-2)-c_{1}(1+|A|)-c_{2}\sqrt{2}|A|)\big)\Theta_{\min}\\
         &=&\big(|A|^{2}-(c_{1}+c_{2}\sqrt{2})|A|-(2+c_{1})\big)\Theta_{\min}
\end{eqnarray*}
Since $\Theta_{\min}(0)\geq{}c_0>0$, then above forces $\Theta_{\min} \ge 0$ and hence
\begin{equation*}
 \ddl{}{t} \Theta_{\min} \ge -\left(\frac{(c_{1}+c_{2}\sqrt{2})^{2}}{4}+2+c_{1} \right)\Theta_{\min}.
  \end{equation*}
\begin{equation*}
   \Theta_{\min}(t)\geq{}c_{0}\exp(-(\frac{(c_{1}+c_{2}\sqrt{2})^{2}}{4}+2+c_{1})t) > 0
   \quad\text{on}\ [0,T)\ ,
\end{equation*}
completing the proof.
\end{proof}

\subsection{Uniform bounds on the {\hf}}
In this subsection, we show that the evolving surfaces stay in a compact region in $M^3$ throughout the flow.
This is established via an estimate on the {\hf} $u(x,t)$:
\begin{theorem}\label{bd4hf}
Suppose the mean curvature flow $(1.1)$ has a solution on $[0,T)$, $0<T\leq\infty$, then the {\hf}
$u(\cdot,t)$ (as defined in (4.6)) is uniformly bounded for $t \in [0,T)$.
\end{theorem}
Proposition 3.4 plays a very important role in this key theorem. We use the hyperbolic properties of {\ef}, and
Hopf's {\maxp} to bound {\es}s of the {\mcf}, i.e., hyperbolic geometry provides barrier surfaces for the {\mcf}.
\begin{proof} At each time $t\in[0,T)$, let $x(t)\in{}S(t)$ be the point such that
\begin{equation*}
   u_{\max}(t)\equiv\max_{x\in{}S(t)}u(x,t)=u(x(t),t)\ ,
\end{equation*}
and let $y(t)\in{}S(t)$ be the point such that
\begin{equation*}
   u_{\min}(t)\equiv\min_{y\in{}S(t)}u(y,t)=u(y(t),t)\ .
\end{equation*}
By the {\ee} $(4.9)$ and the positivity of $\Theta$ along the flow (the Lemma \eqref{bd4gf}), we find the part
of $S(t)$ with $H < 0$ will move along the positive direction of $\n$ while the part of $S(t)$ with $H > 0$ will
move along the negative direction of $\n$, therefore we can assume that $u_{\max}(t)$ is increasing and
$u_{\min}(t)$ is decreasing, for $t \ge t_0$, where $t_0>0$.

Our strategy is now clear: in the positive direction, at the furtherest point on $S(t)$, the {\mc} is negative. We
then apply the Proposition 3.4, that $M^3$ admits an {\ef} such that for far enough (at least $r_0$ from the
reference surface), all fiber surfaces have positive {\mc}s. Fiber surfaces at $r_0$ and $-r_0$ then serve as
barrier surfaces for the {\mcf} $(1.1)$ by Hopf's {\maxp}.

We now follow the strategy: Since $M^3$ is {\af}, the parallel surfaces from the {\ms} $\Sigma$ form the {\ef},
$\{\Sigma(r)\}_{r \in \R}$, of $M^3$. Therefore, there exist $r_1 < r_2$ such that the surface $\Sigma(r_1)$ is
tangent to $S(t)$ at the point $y(t)$, and the surface $\Sigma(r_2)$ is tangent to $S(t)$ at the point $x(t)$. It
is easy to see that $u_{\max}(t) = r_2$ and $u_{\min}(t) = r_1$.

Let us assume the reference surface is the {\ms} $\Sigma$ to apply some basic properties of {\pc}s for {\ef}
$\{\Sigma(r)\}_{r \in \R}$, as in Proposition 3.4: from the formula $(2.2)$, the {\pc}s of any point on $\Sigma(r)$ are
determined by the {\pc}s of the corresponding point on $\Sigma$ and $r$. In particular, if $\mu_1(x,r) < \mu_2(x,r)$
are {\pc}s of $(x,r) \in \Sigma(r)$, where $x \in \Sigma$, then we have, as in $(2.2)$:
\begin{equation*}
\mu_{1}(x,r) = \frac{\tanh(r)-\lambda(x)}{1-\lambda(x)\tanh(r)},
\mu_{2}(x,r) = \frac{\tanh(r)+\lambda(x)}{1+\lambda(x)\tanh(r)},
\end{equation*} where $\pm\lambda(x)$ are the {\pc}s on the {\ms} $\Sigma$.

It is routine to verify that for fixed $x$, both $\mu_1(\cdot,r)$ and $\mu_2(\cdot,r)$ are increasing function of $r$,
since $|\lambda(x)| <1$. Let $\lambda_0$ be the maximum of the {\pc}s on the {\ms} $\Sigma$, then
$0 < \lambda_0 < 1$.

Denote the constant $r_0 = {\frac{1}{2}}\log{\frac{1+\lambda_0}{1-\lambda_0}}$. The Proposition 2.4 says that,
for any $r>r_0$, we have $\mu_1(x,r) > 0$ for all $x \in \Sigma$, hence all {\pc}s of the parallel surface
$\Sigma(r)$ are positive. Similarly, for any $-r < -r_0$, all {\pc}s of the parallel surface $\Sigma(-r)$ are negative.

An easy modification to treat the general case when the reference surface $S\not= \Sigma$: since
$|\lambda_{j}(S)| < 1$, we foliate the {\af} manifold $M^3$ by the {\ef} $\{S(r)\}_{r \in \R}$. Take now
$\lambda_0 = \max\{|\lambda(S)|\}$ and $r_0 = {\frac{1}{2}}\log{\frac{1+\lambda_0}{1-\lambda_0}}$. Then the
{\pc}s of $S(r)$ take forms of
\begin{equation*}
      \mu_{j}(x,r) = \frac{\tanh(r)+\lambda_j(x)}{1+\lambda_j(x)\tanh(r)},
\end{equation*}
for $x \in S$ and $j =1,2$. Therefore again we find $-r_0 \le r_1 < r_2 \le r_0$.

We consider, at $F(x(t),t)$, $\Theta=\inner{\n}{\nu}=1$, then $0\leq\ppl{u}{t}=-H$, at $F(x(t),t)$. Therefore the {\mc}
of $S(t)$ at the point $x(t)$ is non-positive. Since $\Sigma(r_2)$ is tangent to $S(t)$ at the positive side of $S(t)$,
by Hopf's {\maxp} (the Lemma 2.1), the {\mc} of $\Sigma(r_2)$ at $x(t)$ is no greater than that of the {\mc} of $S(t)$
at the intersection point $x(t)$. Therefore, there exists at least one point on the surface $\Sigma(r_2)$ with
non-positive {\pc}.So the {\es}s are uniformly bounded by two surfaces $S(-r_0)$ and $S(r_0)$, for all $t \in [t_0, T)$.

Combining with the bounds for $t \in [0,t_0]$, the {\hf} is bounded in a compact region only depending on the reference
surface $S$ and initial surface $S_0$.
\end{proof}
As a corollary, since all evolving surfaces are staying within a compact region, we find:
\begin{cor}
There is a constant $c_3 > 0$, only depending on the initial surface $S_0$ and the {\ms} $\Sigma$, such that
$c_3 \le \Theta(\cdot,t) \le 1$, for all $t \in [0,T)$.
\end{cor}


\section{Proof of Theorem 1.1: long time solution}
We conclude the proof of the Theorem 1.1 in the section. To prevent singularity occurs in finite time, we need to
derive uniform bounds for the {\sff}s of the {\es}s. We use the uniqueness of the {\ms} in an {\af} manifold to show
the {\mcf} equation $(1.1)$ converges to a unique limiting surface.


\subsection{Uniform bounds on $|A|^2$ and $H^2$}
A crucial part of proving long time existence of the {\mcf} equation $(1.1)$ is to establish a priori bounds for the
{\sff}s on $S(t)$. In this subsection, we obtain such a bound. As a corollary, we obtain a uniform upper bound for
the square of the {\mc}, $H^2(\cdot,t)$.

\begin{theorem} \label{bd4a}
Suppose the {\mcf} $(1.1)$ has a solution for $t \in [0,T)$, then there is a constant $c_4 > 0$, only depending
on $S_0$, such that $|A|^{2}(\cdot, t) \le c_4 < +\infty$, for all $t \in [0, T)$.
\end{theorem}
\begin{proof}
The strategy is to add enough negative terms to $(4.5)$. To this end, we derive the evolution equation satisfied by
$\eta(\cdot,t) = {\frac{1}{\Theta}}$. This function is well-defined by the positivity of $\Theta(\cdot,t)$ (Lemma 4.4).
\begin{eqnarray*}
\ppl{}{t}\eta & = & -\eta^2\ppl{}{t}\Theta \\
&=& \Delta\eta - 2\Theta|\nabla\eta|^2 -(|A|^{2}-2)\eta -\eta^2 J,
\end{eqnarray*}
where we denote $J =\n(H_{\n})-H\inner{\nablabar_{\nu}\n}{\nu}$.

We consider the function $f(\cdot,t) = |A|^2 \eta^4$, and we have
\begin{equation}
   \ppl{}{t}f = \eta^4(\Delta|A|^2 - 2|\nabla A|^2 + 2|A|^2(|A|^2+2) - 4H^2) + |A|^2\ppl{}{t}(\eta^4),
\end{equation}
and we compute
\begin{equation*}
   \ppl{}{t}(\eta^2) = 2\eta \ppl{}{t}\eta = 2\eta\Delta\eta -4|\nabla\eta|^2 -2\eta^2(|A|^2-2) -2\eta^3 J,
\end{equation*}
therefore we obtain:
\begin{equation}
   L(\eta^2) = -6|\nabla \eta|^2 - 2\eta^2(|A|^2-2) -2\eta^3 J,
\end{equation}
where we introduce the operator $L =  \ppl{}{t}-\Delta$ to simplify our notation.

Therefore we have
\begin{eqnarray}
\ppl{}{t}(\eta^4) &=& 2\eta^2(\Delta(\eta^2) -6|\nabla \eta|^2 - 2\eta^2(|A|^2-2) -2\eta^3 J) \nonumber \\
&=& \Delta(\eta^4) -20\eta^2|\nabla \eta|^2 - 4\eta^4(|A|^2-2) -4\eta^5 J.
\end{eqnarray}
Applying this into $(5.1)$, we find
\begin{eqnarray*}
   L(f) &=& -2\eta^4|A|^4 -2\eta^4|\nabla A|^2 +4(3|A|^2 -H^2) \eta^4 \\
   &  &- 2\nabla|A|^2 \cdot \nabla(\eta^4) -20|A|^2\eta^2|\nabla \eta|^2 -4|A|^2\eta^5 J,
\end{eqnarray*}
where the dominant term on the right is $-2\eta^4|A|^4 = -2\Theta^4 f^2$, for large $|A|^2$.

Now we assume $|A|^2$ is not uniformly bounded, then $|A|_{\max}^2(\cdot,t) \to \infty$ as $t \to T$.
Since $f(\cdot,t) = |A|^2 \eta^4 \ge |A|^2$, we have $f_{\max}(\cdot,t) \to \infty$ as $t \to T$. There
then exists a $t_1 \in (0,T)$ such that when $t > t_1$, we have
\begin{eqnarray*}
   \ddl{}{t}f_{\max} &\le& -2\eta^4|A|^4 -2\eta^4|\nabla A|^2 + 4(3|A|^2 -H^2) \eta^4 \\
   &  & - 2\nabla|A|^2 \cdot \nabla(\eta^4) -20|A|^2\eta^2|\nabla \eta|^2 -4|A|^2\eta^5 J.
\end{eqnarray*}
From Corollary 4.6, we have $-2\eta^4|A|^4 = -2\Theta^4 f^2 \le -2c_3^4 f^2$.

From the proof of Lemma 4.4, we estimate the term $J$:
\begin{equation*}
|J| = |\n(H_{\n})-H\inner{\nablabar_{\nu}\n}{\nu}| \le c_1\Theta^3+(c_1|A|+c_2|H|)\Theta^2.
\end{equation*}
Since $c_3 \le \Theta < 1$, for $t>t_1$, now we have
\begin{equation*}
   \ddl{}{t}f_{\max} \le -2c_3^4 f_{\max}^2 + lower \ order \ terms.
\end{equation*}
This is a contradiction since $\ddl{}{t}f_{\max} >0$. Therefore $f(\cdot,t) = |A|^2 \eta^4$ is uniformly
bounded, which in turn bounds $|A|^2$ from above.
\end{proof}
As a corollary, we obtain the uniform bound for $H^{2}(\cdot, t)$:
\begin{cor}
Let $c_4 > 0$ be the constant as the upper bound of $|A|^2$ in Theorem 4.1, then
$sup_{S(t)}(H^2(\cdot, t)) \le 2c_4$.
\end{cor}
\begin{proof}
This immediately follows from $|A|^2 - {\frac{1}{2}}H^2 \ge 0$.
\end{proof}

\subsection{Long time solution}
We prove the {\mcf} equation $(1.1)$ admits long time solution, i.e., $T = +\infty$, in this subsection:
\begin{theorem}
The maximal time of existence of the solution to the equation $(1.1)$ is $T = +\infty$.
\end{theorem}
We start with controlling the derivatives for the {\hf} along the {\mcf}:
\begin{lem}\label{bd4hfd}
If the {\mcf} equation $(1.1)$ has a solution on $[0,T)$, $0<T\leq\infty$, then
\begin{equation*}
   |\nabla^{\ell}u|\leq{}K_{\ell}< \infty ,
\end{equation*}
for all $\ell=1,2,\ldots$, where $\{K_{\ell}\}_{\ell=1}^{\infty}$ is the collection of constants depending
only on the initial data and the maximal time $T$.
\end{lem}
\begin{proof}
Since the evolving surfaces are graphical surfaces, we have, from \cite{Hui86},
$\Theta(\cdot,t)=1/\sqrt{1+|\nabla{}u|^{2}}$. From Corollary 4.6,  there is a positive lower bound for
$\Theta(\cdot, t)$, therefore $|\nabla{}u|$ is uniformly bounded from above by a constant
depending only on the initial data and $T$.

We observe that equation $(4.8)$:
\begin{equation*}
\ppl{}{t}u = -H\Theta =\Delta u -div(\bar{\nabla}\ell)
\end{equation*}
is a single quasilinear parabolic equation for the {\hf} $u(\cdot,t)$, while $u(\cdot,t)$ is uniformly bounded
by the Theorem 4.5. This enables us to apply the standard regularity results in quasilinear second order
parabolic equations (\cite{Fri64}, \cite{Lie96}) to bound all derivatives of $u$.
\end{proof}
Using Lemma 5.4 and the relation $\Theta(\cdot,t)=1/\sqrt{1+|\nabla{}u|^{2}}$, we obtain upper bounds
for the derivatives of the {\gradf} $\Theta(\cdot,t)$:
\begin{cor}
There exist constants $0<K'_{\ell}<\infty$ depending only on $S_{0}$ and $T$ such that
\begin{equation*}
   |\nabla^{\ell}\Theta|^{2}\leq{}K'_{\ell}
\end{equation*}
on $S(t)$ for $0\leq{}t<T$, and $\ell = 1,2,\ldots$.
\end{cor}
Having bounded the {\hf} $u$, the {\gradf} $\Theta$, and their derivatives, we obtain the estimates for the
derivatives of the {\sff}:
\begin{pro}
For each $n \ge 1$, there is a constant $c_5(n,S_0) > 0$ such that $|\nabla^{n}A|^2 \le c_5(n,S_0)$
uniformly on $S(t)$, for $t \in [0,T)$.
\end{pro}
\begin{proof}
The basic {\ee} for $|\nabla^{n}A|^2$ is (\cite{Ham82}, \cite{Hui84}):
\begin{equation*}
   L(|\nabla^{n}A|^2) = -2|\nabla^{n+1}A|^2 + K''_{i,j,k,n}\sum_{i+j+k=n}
   \nabla_{i}A\ast \nabla_{j}A\ast \nabla_{k}A\ast \nabla_{n}A,
\end{equation*}
where $L = \ppl{}{t} - \Delta$, and $K''_{i,j,k,n}$ is a constant depending on $S_0$, and natural numbers $i$,
$j$, $k$, and $n$ with $i+j+k=n$. The upper bounds for all derivatives of $A$ follow from an induction on $n$.
\end{proof}
With all the pieces in place, we conclude this subsection with:
\begin{proof} [proof of Theorem 5.3]
By the Theorem 4.5, the {\hf} $u(\cdot, t)$ is uniformly bounded, therefore, the surfaces $\{S(t)\}_{t\in[0,T)}$
stay in a compact smooth region in $M^3$. Applying the Theorem 5.1, Corollary 5.2, Corollary 5.5 and
Proposition 5.6, the sequence $\{S(t)\}$ converges to a limiting smooth surface $S_T$ as $t \to T$, which is
again, a graph. Now we use $S_T$ as our initial surface in the equation $(1.1)$ to extend the flow beyond
$T$, by the existence of short-time solutions result for the new parabolic equation (Theorem 2.3).
\end{proof}

\subsection{Convergence and uniqueness}
In this subsection, we conclude the proof of our Theorem 1.1 by showing the convergence and uniqueness
of the limiting surface to the equation $(1.1)$. We first establish the {\mc} estimate for the limiting surface.
\begin{theorem}
Let $H(\cdot,t)$ be the {\mc} functions of the evolving surfaces $S(t)$, then $Sup_{S(t)}|H(\cdot,t)| \to 0$ as
$t \to \infty$.
\end{theorem}
\begin{proof}
We consider the function $h(t) = \int_{S(t)}H^{2}d\mu$. From the formula $(4.3)$, we have
\begin{equation}
   \ddl{}{t}|S(t)|= -h(t) \le 0,
\end{equation}
where $|S(t)|$ is the surface area of $S(t)$.

We integrate both sides of $(5.4)$, since $T = \infty$, to find
\begin{equation*}
  \int_0^{\infty}h(t)dt = |S_0| - |S(\infty)| \le  |S_0|.
\end{equation*}

Because the initial surface $S_0$ is a closed {\is}, we know that $|S_0|$ is finite, hence the function $h(t)$ is
uniformly bounded.

We now compute the $t-$derivative of this function $h(t)$:
\begin{eqnarray*}
\ddl{}{t} h & = & \ddl{}{t}\int_{S(t)}H^{2}d\mu \\
& = & \int_{S(t)}2H H_t-H^4 d\mu \\
 & = & \int_{S(t)}-2|\nabla H|^2+2H^2(|A|^2-2) - H^4 d\mu.
\end{eqnarray*}
Note that we also have uniform bound for $|\nabla H(\cdot, t)|$, which follows from $(4.2)$, Theorem 4.5, and
Corollary 5.5, or Poposition 4.6. Applying this and upper bounds on $|A|^2$ (Theorem 5.1), we find that the
term $|\ddl{}{t} h(t)|$ is also uniformly bounded in $t$. Now the function $h(t)$ is bounded in $t$, with bounded
derivative in $t$, therefore it must tend to zero as $t \to \infty$.

Now a standard interpolation argument, with the uniform bounds on $H$ and $|\nabla H(\cdot, t)|$, shows
that $Sup_{S(t)}|H| \to 0$ as $t \to \infty$.
\end{proof}
\begin{proof} (of Theorem 1.1) Part (i) is proved by Theorem 5.3, and part (ii) is implied by Lemma 4.4. We are
left to show convergence and properties of the limiting surface.

From Theorem 5.7, every sequence $S(t)$ has a subsequence $S(t_i)$ converging smoothly to some stationary
asymptotic limiting surface, say $S_i(\infty)$. It is clear that this limiting surface is closed and incompressible which
satisfies that its {\mc} is zero, hence minimal. However, the ambient space is {\af}, i.e., $M^3$ only admits one
unique incompressible {\ms}, which is $\Sigma$.

The convergence is exponential. This can be seen as follows: for $t$ large enough, since {\es}s are closed, and
the {\ms} has {\pc}s of absolute values less than one, we have $|A|^2 \le 2 -\delta$, for some constant $\delta > 0$.
Then we find, in the proof of the Theorem 5.7: $\ddl{}{t} h(t) \le -2\delta h(t)$, which forces the exponential
convergence.
\end{proof}

\section{Applications}
In this section, we apply our method to more general cases. Applications include other class of hyperbolic
{\tm}s, as well as more geometrical properties of the {\mcf} if better initial data is chosen.
\subsection{Nearly Fuchsian manifolds}
In this subsection, we consider a slightly larger class of {\qf} manifolds than {\af} manifolds, i.e.,
\begin{Def}
A {\qf} manifold $N$ is called {\bf {\naf}} if it contains a closed {\is} (not necessarily minimal) of {\pc}s within
the unit interval $(-1,1)$.
\end{Def}
It is not known if any {\naf} manifold is in fact {\af}.

From the proof of Theorem 1.1, after a modification of the argument, it is easy to see that parts (i) and (ii) of
Theorem 1.1 hold in this class:
\begin{theorem}
Let $N$ be a {\naf} {\tm} which contains a closed {\is} $S$ whose {\pc}s lying in $(-1,1)$. If a smooth closed surface
$S_0$ in $N$ is a graph over $S$: there is a constant $c_0 > 0$ such that $\langle{\n},{\nu_0}\rangle \geq c_0$,
where $\n$ and $\nu_0$ are as in Theorem 1.1. Then:
\begin{enumerate}
\item
the {\mcf} equation $(1.1)$ with initial surface $S(0) = S_0$ has a long time solution;
\item
the {\es}s $\{S(t)\}_{t \in \R}$ stay smooth and remain as graphs over $S$ for all time;
\item
For each sequence $\{ t \}$, there is a subsequence $\{t'\}$ such that {\es}s $\{S(t')\}_{t'  \to \infty}$ converge
smoothly to a {\ms} $S(\infty)$, which is embedded.
\end{enumerate}
\end{theorem}
The flow limit $S(\infty)$ is unique for a given {\ins} (\cite {Sim83}). However, we note that in comparison with
Theorem 1.1, without assuming $N$ is {\af}, it is possible that $N$ might contain several {\ms}s since different {\ins}s
may be deformed into different {\ms}s. In \cite{Wan10}, certain {\qf} manifolds contain several {\ms}s.

We note that Theorem 1.3 also holds for nearly Fuchsian three-manifolds, i.e., if $M^3$ is {\naf}, and $S$ is a closed 
{\is} with maximal absolute {\pc} $\lambda_0 = \max\{|\lambda(S)|\} < 1$, then the {\hd} of the limit set associate to 
$M^3$ is less than $1+\lambda_0^2$.
\subsection{Mean-convexity}
Theorems 1.1 and 6.2 indicate that the {\mcf} $(1.1)$ is quite insensitive about the initial data. As an application,
we consider more specific initial surfaces, i.e., mean-convex surfaces:
\begin{theorem}
Let $N$ be a {\naf} {\tm} with a closed {\is} $S$ whose {\pc}s lying in $(-1,1)$. If a smooth closed surface
$S_0$ in $N$ is a graph over $S$ such that $S_0$ has positive {\mc} everywhere. Let $\{S(t)\}_{t \in \R}$ be
{\es}s for the {\mcf} equation $(1.1)$ with initial surface $S(0) = S_0$, then:
\begin{enumerate}
\item
$H(S(t)) > 0$ for all $t \in \R$;
\item
the {\hf} $u(\cdot, t)$ is decreasing in $t$.
\end{enumerate}
\end{theorem}
\begin{proof}
Part (ii) is an easy consequence of part (i) and the equation $(4.8)$. We only show {\mcf} preserves the positivity
of {\mc}s.

The existence of such surfaces of positive {\mc} everywhere is obvious: by Proposition 3.4 and proof of
Theorem 4.5, there is a $r_0 > 0$, only depending on $S$, such that each parallel $\{S(r)\}_{r > r_0}$ to $S$ has
positive {\pc}s everywhere. We might just take any such $S(r)$ as the initial surface $S_0$ for the {\mcf} $(1.1)$.

Let $H_{\min}(\cdot,t)$ be the minimal {\mc} of the surface $S(t)$. Recall from the {\ee} $(4.4)$ for $H(\cdot, t)$, i.e.,
\begin{equation*}
H_t =\Delta H + H(|A|^2 -2).
\end{equation*}
And our initial surface satisfies $H(\cdot, 0) > 0$ by assumption. Let $t_0$ be the first time that $H_{\min}(t_0) = 0$,
then at $t_0$ and at where $H_{\min}$ occurs, we have
\begin{equation*}
\ddl{}{t}|_{t=t_0} H_{\min} \ge  H_{\min}(t_0) (|A|^2 -2) = 0.
\end{equation*}
Here we also note the fact that the surface $S(t_0)$ is closed. This now implies that $H \ge 0$ for all time so long
the solution to the {\mcf} exists. This means we can refine the estimate to, at $t_0$ and $H_{\min}$,
\begin{equation*}
\ddl{}{t}|_{t=t_0} H_{\min} \ge  H_{\min}(t_0) (|A|^2 -2) \ge -2H_{\min},
\end{equation*}
which we arrive at the conclusion $H_{\min}(t) \ge H_{\min}(0)e^{-2t} > 0$.
\end{proof}
Similar results hold for {\mcf} with initial surface $S'_0$ of negative {\mc} everywhere, with reversed inequality for
the {\hf}.

\bibliographystyle{amsalpha}
\bibliography{ref-deform}
\end{document}